\documentclass[10pt,reqno]{amsart}
\usepackage{amsmath,amsthm,amsfonts,amssymb}
\usepackage{color,comment,soul}
\usepackage{tikz}

\newcommand{\R}{\mathbb{R}}
\newcommand{\C}{\mathbb{C}}
\newcommand{\herm}[1]{H_{#1}(\C)}

\newcommand{\CP}{\mathbb{C}P}
\newcommand{\tr}{\mathrm{tr}\,}
\newcommand{\re}{\mathrm{Re}\,}
\newcommand{\im}{\mathrm{Im}\,}
\renewcommand{\span}{\text{Span}\,}
\newcommand{\inter}{\text{int}\,}
\newcommand{\conv}{\text{Conv}}

\newtheorem{theorem}{Theorem}
\newtheorem{lemma}{Lemma}

\theoremstyle{definition}

\newtheorem*{remark}{Remark}

\begin{document}
\title[Continuous Selections of the Inverse Numerical Range Map]{Continuous Selections of the Inverse Numerical Range Map}
\author[B. Lins, P. Parihar]{Brian Lins$^*$, Parth Parihar}
\date{}
\address{Brian Lins, Hampden-Sydney College}
\email{blins@hsc.edu}
\thanks{$^*$Corresponding author.}
\address{Parth Parihar, Princeton University}
\email{pspariha@princeton.edu}
\subjclass[2000]{Primary 15A60, 47A12; Secondary 54C08}
\keywords{Field of values; numerical
range; inverse continuity; continuous selections}
\thanks{This work was partially supported by NSF grant DMS-0751964}

\begin{abstract}
For a complex $n$-by-$n$ matrix $A$, the numerical range $F(A)$ is the range of the map $f_A(x) = x^*A x$ acting on the unit sphere in $\C^n$. We ask whether the multivalued inverse numerical range map $f_A^{-1}$ has a continuous single-valued selection defined on all or part of $F(A)$.  We show that for a large class of matrices, $f_A^{-1}$ does have a continuous selection on $F(A)$. For other matrices, $f_A^{-1}$ has a continuous selection defined everywhere on $F(A)$ except in the vicinity of a finite number of exceptional points on the boundary of $F(A)$.  
\end{abstract}
\maketitle

\section{Introduction}

The \textit{numerical range} (also known as the \textit{field of values}) $F(A)$ of a matrix $A \in M_n(\C)$ is the image of the complex unit sphere $\C S^n = \{x \in \C^n : x^*x = 1\}$ under the map $f_A(x) = x^*Ax$. Since the map $f_A(x)$ is continuous, the numerical range is always a compact, connected subset of $\C$.  It contains the eigenvalues of $A$ and is convex \cite{Davis71}.  These properties make the numerical range a useful tool in applications and within linear algebra.   

Recently, several papers have studied the pre-images of a complex number $z \in F(A)$ under the map $f_A$. For any unimodular constant $\omega \in \C$ and $x \in \C S^n$, $f_A(\omega x) = f_A(x)$. It follows that the pre-images of $f_A$ will always be trivially multivalued. In fact, $f_A^{-1}(z)$ contains a set of $n$ linearly independent vectors for every $z \in \inter F(A)$ \cite[Theorem 1]{Carden09}.  Algorithms for computing at least one element of $f_A^{-1}(z)$ are presented in \cite{Carden09,CPU2010,Meurant2012,Uhlig2008}. 

As a multivalued map, there are several possible notions of continuity that could apply to the inverse numerical range map $f_A^{-1}$.  In \cite{CJKLS2} the following definitions were introduced.  Let $g$ be a multivalued mapping from a metric space $(X,d_X)$ to a metric space $(Y,d_Y)$.  We say that $g$ is \textit{weakly continuous} at $x \in X$ if for all sequences $x_k \rightarrow x$ in $X$, there exists $y \in g(x)$ and a sequence $y_k \in g(x_k)$ such that $y_k \rightarrow y$.  If such sequences exist for all $y \in g(x)$, then $g$ is \textit{strongly continuous} at $x$. Alternatively, $f_A^{-1}$ is weakly (strongly) continuous at $z\in F(A)$ if the direct mapping $f_A$ is open with respect to the relative topology on $F(A)$ at some (resp., all) pre-images $x\in f_A^{-1}(z)$. In \cite{CJKLS2}, it was shown that the inverse field of values map is strongly continuous on the interior of $F(A)$, and that strong continuity can only fail at so-called round points of the boundary. Necessary and sufficient conditions for weak and strong continuity of $f_A^{-1}$ are given in \cite{LLS2}. In particular, strong (and therefore weak) continuity can only fail at finitely many points on the boundary \cite[Corollary 2.3]{LLS2}.

Related notions of continuity for multivalued maps include the notions of upper and lower semi-continuity \cite{Michael1}.  As the inverse of a continuous single-valued function, $f_A^{-1}$ is automatically upper semi-continuous. In our terminology, strong continuity is equivalent to lower semi-continuity.  

Given any multivalued function $g:X \rightrightarrows Y$, we may also ask whether there exists a continuous single valued function $h: X \rightarrow Y$ such that $h(x) \in g(x)$ for all $x \in X$.  Such a function is called a \textit{continuous selection} of $g$. There are several general theorems due to Michael \cite{Michael1,Michael2} concerning whether upper and lower semi-continuous multivalued functions admit a continuous single-valued selection.  These theorems require additional convexity or connectedness assumptions in order to apply.  In the case of the map $f_A^{-1}$ the convexity assumptions do not apply, and the connectedness assumptions are difficult to verify. By the definition of weak continuity, if $G$ is a relatively open subset of $F(A)$ containing a point where $f_A^{-1}$ is not weakly continuous, then it is not possible to define a continuous selection of $f_A^{-1}$ on $G$. 

The main result of this paper is the following theorem.  

\begin{theorem} \label{thm:main}
Let $A \in M_n(\C)$. If $f_A^{-1}$ has no weak continuity failures on $F(A)$, then there is a continuous selection of $f_A^{-1}$ on all of $F(A)$.  If $f_A^{-1}$ has weak continuity failures at $w_1, \ldots, w_k \in \partial F(A)$, then for any open set $G$ containing $\{w_1, \ldots, w_k\}$ there is a continuous selection of $f_A^{-1}$ on $F(A) \backslash G$.  
\end{theorem} 

Note that $f_A^{-1}$ is weakly continuous everywhere on $F(A)$ when $A$ is normal and when $n \le 3$ \cite[Corollaries 5 and 6, and Theorem 11]{CJKLS2}. For such matrices it follows that $f_A^{-1}$ has a continuous selection defined on all of $F(A)$. More generally, the set of matrices $A \in M_n(\C)$ for which $f_A^{-1}$ is strongly (and therefore also weakly) continuous on all of $F(A)$ is generic \cite[Proposition 2.4]{LLS2}.

\section{Preliminaries}
When considering the map $f_A$, it is natural to identify vectors that are scalar multiples in $\C S^n$.  Under this identification, $\C S^n$ becomes the complex projective space $\C P^{n-1}$. With the inclusion map $q: \C S^n \rightarrow \C P^{n-1}$, there is a unique map $\hat{f}_A: \C P^{n-1} \rightarrow F(A)$ that makes the diagram below commute.  

\begin{center}
\begin{tikzpicture}
\fill (0,0) node (A) {$\C S^n$};
\fill (5,0) node (B) {$F(A)$};
\fill (0,-3) node (C) {$\CP^{n-1}$};
\draw[thick,->] (A) to (B);
\draw[thick,->] (A) to (C);
\draw[thick,dashed,->] (C) to (B);
\draw (2.5,0) node[above] {$f$};
\draw (0,-1.5) node[left] {$q$};
\draw (2.5,-1.5) node[below right] {$\hat{f}$};
\end{tikzpicture}
\end{center}

Importantly, even the map $\hat{f}_A$ has multivalued pre-images since $f_A^{-1}(z)$ contains $n$ linearly independent vectors in $\C S^n$ when $z$ is in the interior of $F(A)$ \cite[Theorem 1]{Carden09}. On the boundary of $F(A)$, $\hat{f}_A^{-1}$ may be single or multivalued \cite{LLS1}. 

The complex projective space $\CP^{n-1}$ is homeomorphic to the set $\{xx^*:x \in \C S^n\}$ via the map $\varphi:[x] \rightarrow xx^*/x^*x$. It will be convenient to use $\{xx^*:x \in \C S^n\}$ as a representation for $\CP^{n-1}$, so when we write $\CP^{n-1}$, we mean the set $\{xx^*:x \in \C S^n\}$.  Thus, $\C P^{n-1} \subset H_n(\C)$, where $H_n(\C)$ denotes the set of $n$-by-$n$ complex Hermitian matrices.  Note that $\hat{f}_A(xx^*) = \tr(Axx^*)$, so $\hat{f}_A$ is the restriction to $\CP^{n-1}$ of a real linear map from the real vector space $H_n(\C)$ to $\C$.    

In order to present a continuous selection of $f_A^{-1}$ on $V \subseteq F(A)$, it is sufficient to find a subset $U \subset \C S^n$ such that $f_A$ is a bijection from $U$ to $V$.  It follows immediately from the compactness of $\C S^n$ that $f_A^{-1}$ is a continuous function from $V$ to $U$, so by restricting the range to $V$, we obtain a continuous selection. We may also find a subset $W \subset \C P^{n-1}$ on which $\hat{f}_A$ is a bijection onto $V$, in order to arrive at a continuous selection of $\hat{f}_A^{-1}$ on $V$. Our approach to finding sets on which $f_A$ (or $\hat{f}_A$) is a bijection is to parametrize a subset $U \subset \C S^n$ ($W \subset \C P^{n-1}$) and show that this parametrization composed with $f_A$ (respectively, $\hat{f}_A$) is a parametrization of the corresponding $V \subset F(A)$.   

For any matrix $A \in M_n(\C)$, recall the real and imaginary parts of $A$, $\re(A) = (A+A^*)/2$ and $\im(A) = (A-A^*)/2i$. Given $A \in M_n(\C)$, note that $F(e^{-i\theta}A) = e^{-i\theta}F(A)$ for any $\theta \in [0,2\pi)$.  Furthermore, $e^{-i\theta}A = \re (e^{-i \theta} A) + i \im (e^{-i \theta} A)$.  The left and right-most points on the rotated numerical range $F(e^{-i\theta} A)$ correspond to the maximal and minimal eigenvalues of the matrix $\re (e^{-i \theta} A)$. The map $\theta \mapsto \re (e^{-i \theta}A)$ is an analytic self-adjoint matrix valued function.  By a theorem of Rellich \cite[Corollary 2, section 3.5.5]{Baumgartel}, there is a family of $n$ functions $x_1(\theta), \ldots, x_n(\theta)$ that are analytic in $\theta$ on $[0,2\pi]$ and form an orthonormal basis of eigenvectors for $\re(e^{-i\theta}A)$.  For each $j \in \{1,\ldots,n\}$ let $\lambda_j(\theta)$ denote the eigenvalue of $\re(e^{-i\theta}A)$ corresponding to $x_j(\theta)$.   Of course, the eigenvalue-functions $\lambda_j(\theta)$ are also analytic in $\theta$ since $\lambda_j = x_j^*\re(e^{-i \theta}A)x_j$.  

For each eigenvector-function $x_j(\theta)$ there is an associated \emph{critical curve}, defined by $z_j(\theta) = f_A(x_j(\theta))$.  The images of these critical curves are contained in the numerical range $F(A)$. Furthermore, $F(A)$ is precisely the convex hull of the critical curves.  Using the fact that $\frac{d}{d\theta} \re(e^{-i \theta} A) = i \im(e^{-i \theta}A)$, we can derive the following alternative formula for the critical curves,  
\begin{equation} \label{eq:critcurves}
z_j(\theta) = e^{i \theta}(\lambda_j(\theta) + i \lambda_j'(\theta)).  
\end{equation}
The relationship between a numerical range and its critical curves was first described by Kippenhahn \cite{Ki}. See also \cite{JAG98} and \cite[Section 5]{GaSe12} for more details.   

From \eqref{eq:critcurves}, it is clear that the eigenvalue-functions of $\re(e^{-i \theta}A)$ determine the shape of the numerical range.  Since the eigenvalue-functions $\lambda_j$ are analytic, any two of them may only coincide at finitely many values of $\theta$ unless they are identical.  Thus, for all but finitely many angles $\theta \in [0,2\pi)$, $\re (e^{-i\theta} A)$ has $m \le n$ distinct eigenvalues. We will call $\theta$ an \emph{exceptional argument} if two or more distinct eigenvalue-functions coincide at $\theta$. The corresponding points $z_j(\theta)$ given by \eqref{eq:critcurves} are \emph{exceptional points}.  

At an exceptional argument $\theta_0$, there will be at least two distinct eigenvalue-functions achieving the same value.  Since both eigenvalue-functions are analytic, both functions have Taylor series expansions about $\theta_0$, which must differ.  We say the the eigenvalue-functions \emph{split at degree $k$} if the first coefficient where the Taylor series differ is in the degree $k$ term.  Note that flat portions of the boundary of $F(A)$ occur at exceptional arguments where the maximal eigenvalue-functions split at degree one.  This follows immediately from \eqref{eq:critcurves} and the fact that $F(A)$ is the convex hull of the critical curves.  A corner point of $F(A)$ is a point contained at the intersection of two flat portions.  A boundary point that is not a corner point and is not contained in the relative interior of a flat portion will be called a \emph{round point}.  If a round point is not an endpoint of a flat portion, then it is \emph{fully round}.  

The following theorem makes the relationship between the eigenvalue-functions of $\re(e^{-i\theta}A)$ and continuity failures of $f_A^{-1}$ clear.  

\begin{theorem}[Theorem 2.1, \cite{LLS2}] \label{thm:characterization}
Let $A \in M_n(\C)$ and $z = z_j(\theta_0) \in \partial F(A)$.
\begin{enumerate}
\item $f_A^{-1}$ is strongly continuous at $z$ if and only if $z$ is in the relative interior of a flat portion of the boundary or the eigenvalue-functions corresponding to $z$ at $\theta_0$ do not split.  
\item $f_A^{-1}$ is weakly continuous at $z$ if and only if $z$ is in a flat portion of the boundary or the eigenvalue-functions corresponding to $z$ at $\theta_0$ do not split at odd powers. {So weak continuity fails if and only if $z$ is a fully round boundary point and the eigenvalue-functions corresponding to $z$ at $\theta_0$ split at an odd power.}
\end{enumerate}
\end{theorem}

An immediate consequence of Theorem \ref{thm:characterization} is that strong (and therefore weak) continuity of $f_A^{-1}$ can only fail at exceptional points on the boundary of the numerical range. In particular, there are at most finitely many exceptional points where weak continuity can fail.

\section{Continuous Selections on the Boundary}
\begin{lemma} \label{lem:analyticArcs}
For any analytic curve $\Gamma \subseteq \partial F(A)$, there is an analytic path $x:[0,1] \rightarrow \C S^n$ such that $f_A(x(t))$ parametrizes $\Gamma$.  If $F(A)$ has nonempty interior and $\Gamma$ is the whole boundary of $F(A)$, then $x$ may be chosen to be periodic on $[0,1]$.  The intersection of $f_A^{-1}$ with the range of $x(t)$ is a continuous selection of $f_A^{-1}$ on $\Gamma$.  
\end{lemma}

\begin{proof}
The analytic curve $\Gamma$ is either contained in one of the critical curves of $F(A)$, or it is contained in a flat portion of the boundary. In the former case, there is one eigenvalue-function $\lambda(\theta)$  corresponding to the maximal eigenvalue of $\re( e^{-i \theta} A) = \cos(\theta) H + \sin(\theta) K$ such that $\Gamma$ is parametrized by \eqref{eq:critcurves} for $ \theta$ in some closed interval $I \subseteq [0, 2 \pi]$.  Let $P(\theta)$ denote the corresponding spectral projection, which is also an analytic function of $\theta$ (\cite[Theorem II.6.1]{Kato}). 
The map $\varphi:(v,\theta) \mapsto v - P(\theta)v$ has differentials with real rank at most $2n-1$, so by Sard's theorem \cite{sard1942} the range of $\varphi$ must have measure zero in $\C^n$. Choose a $w\in \C^n$ that is not in the range of $\varphi$.  Note that $P(\theta)w \neq 0$ for all $\theta \in I$, otherwise $w - P(\theta)w = w$ which would contradict our assumption that $w$ is not in the range of $\varphi$. Now let $x(\theta) = P(\theta)w/||P(\theta)w||$. By construction, $x(\theta)$ is a unit eigenvector of $\re(e^{-i\theta}A)$ corresponding to the maximal eigenvalue $\lambda(\theta)$, and therefore $f_A(x(\theta))$ parametrizes $\Gamma$.  We can make a simple affine linear change of variables to replace $x(\theta)$ defined on $I$ with $x(t)$ defined on $[0,1]$. If $\Gamma$ is a closed loop, then the corresponding spectral projection $P(\theta)$ is periodic on $[0,2\pi]$. It follows that our construction of $x(t)$ is periodic.  

If $\Gamma$ is a subset of a flat portion of the boundary, then let $x, y \in \C S^n$ be preimages under $f_A$ of the two endpoints of the flat portion. There is an angle $\theta \in [0,2\pi)$ such that $x,y$ are both eigenvectors of $\re(e^{-i\theta}A)$ corresponding to the maximal eigenvalue.  If $A_2$ denotes the compression of $A$ corresponding to the 2-by-2 subspace $\span \{x,y \}$, then $x,y$ are the eigenvectors of $\im (e^{-i\theta}A_2)$ corresponding to the maximal and minimal eigenvalues.  This implies that $x$ and $y$ are orthogonal, and direct computation shows that the flat portion can be parametrized by $f_A(x \cos \omega + y \sin \omega)$ for $\omega \in [0, \pi/2]$.  If we denote $x(\omega) = x \cos \omega + y \sin \omega$, we can make a simple change of variables so that $x(t)$ has domain $[0,1]$.  
\end{proof}

If $f_A^{-1}$ has weak continuity failures on $\partial F(A)$, then it will not be possible to choose a continuous selection on the whole boundary.  If there are no weak continuity failures on the boundary, then it is possible to find a continuous selection as the following lemma proves.  

\begin{lemma} \label{lem:bdrySelect}
Suppose that $f_A^{-1}$ is weakly continuous on all of $F(A)$. Then there is a continuous selection of $f_A^{-1}$ on $\partial F(A)$.   
\end{lemma}

\begin{proof}
In the case where $F(A)$ has no flat portions, the boundary is given by a single critical curve, and Lemma \ref{lem:analyticArcs} immediately implies that there is an analytic, periodic function $x:[0,1] \rightarrow \C S^n$ such that $f_A(x(t))$ parametrizes the boundary. The intersection of $f_A^{-1}$ with the range of $x$ is a continuous selection.  

If there are flat portions, then apply Lemma \ref{lem:analyticArcs} to chose a continuous selection of $f_A^{-1}$ on each curved analytic portion of the boundary. Since there are no weak continuity failures, the maximal eigenvalue-functions that define the boundary of $F(A)$ can only cross at degree one splitting points. Therefore the boundary of $F(A)$ is defined by alternating analytic curves and flat portions. If $F(A)$ has corner points, one or more of the analytic curves may be single points, but that is not a concern.   For a given flat portion, let $x, y \in \C S^n$ be the pre-images of the end points as determined by the continuous selections on the curved portions of $\partial F(A)$. Using Lemma \ref{lem:analyticArcs}, we choose an $x:[0,1] \rightarrow \C S^n$ such that $f_A(x(t))$ parametrizes the flat portion of the boundary.  From the proof of Lemma \ref{lem:analyticArcs}, it is clear that we may choose $x(t)$ such that $x(0) = x$ and $x(1) = y$. Then the map $f_A(x(t)) \mapsto x(t)$ continuously extends our selection of $f_A^{-1}$ to include the flat portion.   
\end{proof}

\color{black}

\color{black}
\section{Constructing the Selection}

In order to construct a continuous selection of $f_A^{-1}$ on the interior of $F(A)$, it will be convenient to have an alternative formula for a selection of $f_A^{-1}$ on $\partial F(A)$. Let $A \in M_n(\C)$ have a numerical range $F(A)$ with non-empty interior.  Suppose that the critical curves corresponding to the maximal eigenvalues of $\re(e^{-i\theta}A)$ cross at the exceptional arguments $\theta_1 < \theta_2 < \ldots < \theta_m$ in $[0,2\pi)$.  By rotation, we may assume without loss of generality that $\theta_1 > 0$.  Fix $x_0 \in f_A^{-1}(z_0)$ where $z_0$ is the point in $F(A)$ with maximal real part.  

On each interval $(\theta_k,\theta_{k+1})$, the spectral projection $P(\theta)$ corresponding to the maximal eigenvalue of $\re (e^{-i\theta} A)$ is an analytic function of $\theta$. The projection valued function $P(\theta)$ on $(\theta_k,\theta_{k+1})$ extends to an analytic function on $\R$ \cite[Theorem II.6.1]{Kato}, which we will denote by $P_k(\theta)$. For all values of $\theta$, $P_k(\theta)$ is a projection into an eigenspace of $\re(e^{-i\theta}A)$, however the corresponding eigenvalue may not be maximal outside $(\theta_k,\theta_{k+1})$. 

The expression $P_k(\theta)x_0$ is analytic, as is $||P_k(\theta)x_0||$.  If $||P_k(\theta)x_0||$ is not identically zero on $[\theta_k,\theta_{k+1}]$, then it must be nonzero for all but finitely many $\theta$ in the interval.  In this case, we claim that there is a piecewise function $\alpha:[\theta_k,\theta_{k+1}] \rightarrow \{-1,1\}$ such that 
\begin{equation} \label{eq:p1}
x_k(\theta) = \alpha(\theta) \frac{P_k(\theta)x_0}{||P_k(\theta)x_0||}
\end{equation} 
has only removable discontinuities, and so we may extend $x_k$ to a continuous function on $[\theta_k,\theta_{k+1}]$. The following lemma proves this claim.  
\begin{lemma} \label{lem:continuize}
Let $I$ be an interval in $\R$, and suppose that $x:I \rightarrow \C^n$ is analytic.  There is a continuous function $y:I \rightarrow \C^n$ such that $y(t) = \pm x(t)/||x(t)||$ for all $t \in I$ where $x(t) \neq 0$.  
\end{lemma}
\begin{proof}
Since $x(t)$ is analytic, so is $x(t)/||x(t)||$, except at points where $x(t) =0$. Let $t_1, \ldots, t_m$ denote these zeros.  Near each zero $t_j$, $x(t)$ has a Taylor series expansion:
$$x(t) = a_1 (t-t_j) + a_2 (t-t_j)^2 + \ldots.$$
Let $k_j$ denote the degree of the first nonzero vector coefficient in the series above.  Note that 
$$\lim_{t \rightarrow t_j^-} \frac{x(t)}{||x(t)||} = \frac{a_{k_j}}{||a_{k_j}||}(-1)^{k_j},$$
and
$$\lim_{t \rightarrow t_j^+} \frac{x(t)}{||x(t)||} = \frac{a_{k_j}}{||a_{k_j}||}.$$
Let $y(t) = c_j \frac{x(t)}{||x(t)||}$ on each open interval between adjacent zeros $t_j$ and $t_{j+1}$, where each $c_j$ is either 1 or -1.  By choosing the $c_j$ constants sequentially, we can ensure that the discontinuities in $y(t)$ at each $t_j$ are removable, and therefore $y(t)$ can be extended to a continuous function on $I$.
\end{proof}

\color{black}
If $P_k(\theta)x_0 = 0$ identically on $[\theta_k,\theta_{k+1}]$, then we choose $w$ as in the proof of Lemma \ref{lem:analyticArcs} such that $P_k(\theta)w \neq 0$ for all $\theta \in [\theta_k,\theta_{k+1}]$. In this case we let 
\begin{equation} \label{eq:p2}
x_k(\theta) = \frac{P_k(\theta)w}{||P_k(\theta)w||}
\end{equation}  
In both \eqref{eq:p1} and \eqref{eq:p2}, we note that  $x_0^*x_k(\theta) \in \R$ for all $\theta \in [\theta_k,\theta_{k+1}]$.  

If the maximal eigenvalue-function $\lambda(\theta)$ of $\re(e^{-i\theta}A)$ splits at even degree at $\theta_k$, then the spectral projections $P_k(\theta)$ and $P_{k-1}(\theta)$ are identical as are the functions $x_k(\theta)$ and $x_{k-1}(\theta)$ (see the proof of Theorem 2.1 in \cite{LLS2}). If $\lambda(\theta)$ splits at degree one then there is a flat portion of the boundary corresponding to the argument $\theta_k$.  On the flat portion, we define the function
\begin{equation} \label{eq:p3}
y_k(\omega) = \cos (\omega) x_{k-1}(\theta_{k}) + \sin (\omega) x_{k}(\theta_k),
\end{equation}  
and we note that $f_A$ is a bijection from the curve $\{ y_k(\omega): \omega \in [0,\pi/2] \}$ in $\C S^n$ onto the flat portion of the boundary corresponding to the argument $\theta_k$.  By traversing the curves $x_k$ and $y_k$ in order and parametrizing the resulting curve with domain $[0,1]$, we obtain a path $y(t)$ in $\C S^n$ such that the image of $f_A(y(t))$ is the boundary of $F(A)$, $x_0^*y(t) \in \R$ for all $t \in [0,1]$, and $y(t)$ is continuous except at values of $t$ where $y(t)$ corresponds to an exceptional argument $\theta_k$ where $\lambda(\theta)$ splits at odd degree greater than 1.  For these $y(t)$, $f_A(y(t))$ is a point on $\partial F(A)$ where weak continuity of $f_A^{-1}$ fails by Theorem \ref{thm:characterization}.  If there are no weak continuity failures on $\partial F(A)$, then $y(t)$ is continuous on $[0,1]$, although $x_0=y(0)$ and $y(1)$ may differ by a constant. We will refer to functions $y(t)$ constructed in this manner  as \textit{canonical selections of $\partial F(A)$}.

\color{black}
\begin{lemma} \label{lem:2by2}
Let $A$ be a non-normal 2-by-2 matrix with complex entries, and suppose that $x, y \in \C S^2$, $x^*y \neq 0$, and $f_A(x)$ an $f_A(y)$ are distinct points in $\partial F(A)$. Let 
\begin{equation} \label{eq:lineselect1}
h(\lambda) = \frac{ \lambda x + (1-\lambda) \left(y^*x + i \beta \sqrt{1-|x^*y|^2} \right) y + \frac{\sqrt{2}}{2}C v}{\sqrt{C+1}}, 
\end{equation}
where $C = 2 \sqrt{\lambda(1-\lambda)(1-|x^*y|)}$, $\beta = \pm \frac{y^*x}{|x^*y|}$, and $v = \frac{\sqrt{2}}{2} \left( x + i \beta  \frac{y-(x^*y)x}{||y-(x^*y)x||} \right)$.  Then 
$$
f_A(h(\lambda)) = \lambda f_A(x) + (1-\lambda)f_A(y).
$$
\end{lemma} 
\begin{proof}
By the well known Elliptical Range Theorem, $F(A)$ is a convex ellipse. An elegant proof of that theorem can be found in \cite{Davis71}. The main observation of \cite{Davis71} is that $\C P^1$ is a 2-sphere of radius $\frac{\sqrt{2}}{2}$ centered at $\frac{1}{2}I_2$ in the affine subspace of $\herm{2}$ consisting of matrices with trace one. Therefore $F(A)$, which is the image of $\C P^1$ under the linear transformation $\hat{f}_A$, must be a convex ellipse.  Since $A$ is not normal, the ellipse is not degenerate (see e.g., \cite{donoghue1957}).  

The following matrices form a basis for the set of trace zero matrices in $\herm{2}$.   
$$X_1 = \begin{bmatrix}
0 & 1 \\ 1 & 0 
\end{bmatrix}, 
X_2 = \begin{bmatrix}
0 & i \\ -i & 0
\end{bmatrix},
X_3 = \begin{bmatrix}
1 & 0 \\ 0 & -1 
\end{bmatrix}.
$$
Consider the linear map $\psi:\herm{2} \rightarrow \R^3$ defined by  
$$\psi(Y) = \begin{bmatrix}
\langle Y, X_1 \rangle \\
\langle Y, X_2 \rangle \\
\langle Y, X_3 \rangle 
\end{bmatrix}.$$
The image of $\C P^1$ under $\psi$ is precisely the unit sphere $S$ in $\R^3$. In fact, if we restrict the domain of $\psi$ to $\C P^1$, then $\psi:\C P^1 \rightarrow S$ is a bijection and for any $r = (r_1,r_2,r_3) \in S$, $\psi^{-1}(r) = \frac{1}{2}(r_1 X_1 + r_2 X_2 + r_3 X_3) + \frac{1}{2}I_2$.  Observe the following relations:
\begin{itemize}
\item For $x, y \in \C S^n$, $\langle xx^*, yy^* \rangle_{\herm{n}} = |x^*y|^2$.
\item For $xx^*, yy^* \in \C P^1$, $\langle xx^*, yy^* \rangle_{\herm{n}} = \frac{1}{2} \langle \psi(xx^*), \psi(yy^*) \rangle_{\R^3} + \frac{1}{2}$.  

\end{itemize}

Let $H = \re(A)$ and $K=\im(A)$ so that $\hat{f}_{A}(Y) = \tr(HY)+i\tr(KY)$. If we identify $\C$ with $\R^2$, then for any $u \in \C S^2$,
$$\hat{f}_{A}(uu^*) = B \psi(uu^*) + \tfrac{1}{2}\tr (A)$$
where 
$$B = \tfrac{1}{2} \begin{bmatrix}
\langle H,X_1 \rangle & \langle H,X_2 \rangle & \langle H,X_3 \rangle \\
\langle K,X_1 \rangle & \langle K,X_2 \rangle & \langle K,X_3 \rangle 
\end{bmatrix}.$$
Since $xx^*$ and $yy^*$ are both mapped to $\partial F(A)$ by $\hat{f}_A$, it follows that  $\psi(xx^*)$ and $\psi(yy^*)$ are both on the great circle of $S$ that is mapped to $\partial F(A)$ by the affine linear transformation $r \mapsto B r + \frac{1}{2} \tr(A)$. The two vectors in $S$ orthogonal to this great circle are in the nullspace of $B$.  Since $|x^*v| = \frac{\sqrt{2}}{2}$ and 
\begin{equation} \label{eq:ystarv}
y^*v = \tfrac{\sqrt{2}}{2}\left( y^*x + i \beta \frac{1 - |x^*y|^2}{||y-(x^*y)x||} \right) = \tfrac{\sqrt{2}}{2}\left( y^*x + i \beta \sqrt{1 - |x^*y|^2} \right),
\end{equation}
so that $|y^*v| = \frac{\sqrt{2}}{2}$, it follows that $\psi(vv^*)$ is orthogonal to both $\psi(xx^*)$ and $\psi(yy^*)$. Since $x^*y \neq 0$, $\psi(xx^*)$ and $\psi(yy^*)$ are not antipodal points on $S$. Therefore $\psi(vv^*)$ must be one of the two vectors in $S$ that are in the nullspace of $B$. 

We now construct a point $s \in S$ such that $Bs + \frac{1}{2} \tr A = \lambda f_A(x) + (1- \lambda) f_A(y)$.  The point $r = \lambda \psi(xx^*) + (1- \lambda) \psi(yy^*)$, has $Br + \frac{1}{2}\tr A =  \lambda f_A(x) + (1- \lambda) f_A(y)$ by construction, but $r \notin S$.  Note that 
$$||r||^2 = \lambda^2 + (1-\lambda)^2 + 2 \lambda(1- \lambda) \langle \psi (xx^*), \psi (yy^*) \rangle_{\R^3} = $$
$$=2 \lambda^2 - 2 \lambda + 1 + 2 \lambda (1- \lambda) (2|x^*y|^2 - 1). $$
Let $C^2 = 1 - ||r||^2$.  Then 
$$C^2 = 2\lambda - 2 \lambda^2 - (2\lambda - 2\lambda^2) (2 |x^*y|^2 - 1) = $$
$$ = 4 \lambda (1- \lambda) (1 - |x^*y|^2).$$
So $C = 2 \sqrt{\lambda(1- \lambda)(1 - |x^*y|^2)}$.  
Since $r \perp \psi(vv^*)$, it follows that $s = r + C \psi(vv^*) \in S$ has the desired properties.  By applying the map $\psi^{-1}$ to $s$, we see that there must be some $u \in \C S^2$ such that 
$$uu^* = \psi^{-1}(s) =  \lambda xx^* + (1-\lambda) yy^* + C (vv^* - \tfrac{1}{2}I_2) \in \C P^1,$$
and $f_A(u) = \hat{f}_A(uu^*) = \lambda f_A(x) + (1- \lambda)f_A(y)$.  Given $uu^*$ it is only possible to determine $u$ up to multiplication by a unimodular constant. However, it is convenient to set 
$$u = \frac{uu^*v}{||uu^*v||} = \frac{uu^*v}{|u^*v|} = \frac{uu^*v}{\sqrt{\langle uu^*, vv^* \rangle_{\herm{2}}}} = \frac{\sqrt{2} uu^*v}{\sqrt{C+1}} = $$
$$ = \frac{ \lambda x + \sqrt{2} (1-\lambda)(y^*v) y + \frac{\sqrt{2}}{2}C v}{\sqrt{C+1}}.$$
By \eqref{eq:ystarv}, $\sqrt{2} y^*v = y^*x + i \beta \sqrt{1 - |x^*y|^2}$. Letting $h(\lambda) = u$ gives
$$h(\lambda) = \frac{ \lambda x + (1-\lambda) \left(y^*x + i \beta \sqrt{1-|x^*y|^2} \right) y + \frac{\sqrt{2}}{2}C v}{\sqrt{C+1}}. $$
\end{proof}

\color{black}

For normal 2-by-2 matrices, the following result can be verified by direct computation.  
\begin{lemma} \label{lem:normal2by2}
Let $A \in M_2(\C)$ be normal with two distinct eigenvalues, and suppose that $x, y \in \C S^2$ are eigenvectors corresponding to the two eigenvalues of $A$. Let
\begin{equation} \label{eq:lineselect2}
h(\lambda) = \frac{ \lambda x + (1-\lambda) i y + \frac{\sqrt{2}}{2}C v}{\sqrt{C+1}}, 
\end{equation}
where $C = 2 \sqrt{\lambda(1-\lambda)}$ and $v = \frac{\sqrt{2}}{2} \left( x + i y \right)$.  Then 
$$
f_A(h(\lambda)) = \lambda f_A(x) + (1-\lambda)f_A(y).
$$
\end{lemma}

Note that \eqref{eq:lineselect2} is equivalent to \eqref{eq:lineselect1} with $\beta = 1$. In fact, any unimodular constant $\beta$ would have worked.

\begin{theorem} \label{thm:nocornermain}
Let $A \in M_n(\C)$ be a matrix such that $F(A)$ has no corner points and $f_A^{-1}$ has no weak continuity failures on $F(A)$. There is a continuous selection $g:F(A) \rightarrow \C S^n$ of $f_A^{-1}$.
\end{theorem}

\begin{proof}
Fix a fully round point $z_0 \in \partial F(A)$ such that the tangent line to $z_0$ in $F(A)$ is not parallel to any flat portions of the boundary of $F(A)$. By rotation, we may assume that $z_0$ is the right-most point in $F(A)$. Since there are no corner points, the boundary of $F(A)$ consists of alternating flat and round portions. On each round portion there is a selection of $f_A^{-1}$ corresponding to one of the curves \eqref{eq:p1} or \eqref{eq:p2}. On each flat portion there is a selection corresponding to \eqref{eq:p3}. By traversing these curves in order as described at the beginning of the section, we may choose a continuous path $y:[0,1] \rightarrow \C S^n$ such that $f_A$ is a bijection from $y([0,1))$ to $\partial F(A)$, $y(0) \in f_A^{-1}(z_0)$, $y(1)$ is a multiple of $y(0)$, and $y(0)^*y(t) \in \R$ for all $t \in (0,1)$. 

Let $x_0 = y(0)$, $y = y(t)$.  For any $z \in F(A) \backslash \{z_0 \}$, there is a unique $\lambda \in [0,1)$ and $t \in (0,1)$ such that $z = \lambda z_0 + (1-\lambda) f_A(y(t))$.  Furthermore, the values of $\lambda$ and $t$ vary continuously in $z$.  We will prove that 
\begin{equation} \label{eq:mainselect}
g(z) = \frac{ \lambda x_0 + (1-\lambda) \left(y^*x_0 + i \sqrt{1-|x_0^*y|^2} \right) y + \frac{\sqrt{2}}{2}C v}{\sqrt{C+1}}, 
\end{equation}
where $C = 2 \sqrt{\lambda(1-\lambda)(1-|x_0^*y|^2)}$ and $v = \frac{\sqrt{2}}{2} \left( x_0 + i  \frac{y-(x_0^*y)x_0}{||y-(x_0^*y)x_0||} \right)$, is a continuous selection of $f_A^{-1}$ on $F(A)$.  Note that $g$ is continuous on $F(A) \backslash \{z_0 \}$ by construction, and it is clear that if we set $g(z_0) = x_0$, then $g$ extends continuously to all of $F(A)$.  All that remains is to prove that $g$ is a selection of $f_A^{-1}$.  

Note that the tangent line to $z_0$ is vertical, and since by assumption there are no vertical flat portions of $\partial F(A)$, there is one $t \in (0,1)$ for which $f_A(y(t))$ is the unique leftmost point of $F(A)$. For all other $t \in (0,1)$, the tangent line to $f_A(y(t))$ in $F(A)$ is not vertical. Let $A_2$ be the 2-by-2 compression of $A$ onto the $\span \{x_0, y \}$.  The numerical range of $A_2$ is a convex ellipse (possibly degenerate) and $z_0, f_A(y) \in F(A_2) \subseteq F(A)$. Recall that $F(A_2)$ is a line segment if and only if $A_2$ is normal.  In that case, $z_0$ and $f_A(y)$ must be the endpoints of that line segment, since $z_0$ and $f_A(y)$ distinct points in the boundary of $A$. The endpoints of the line segment are the eigenvalues of the compression $A_2$, and thus $x_0$ and $y$ are the eigenvectors corresponding to those eigenvalues. Applying Lemma \ref{lem:normal2by2} to the compression $A_2$ shows that $g$ is a continuous selection of $f_A^{-1}$ on the line segment from $z_0$ to $f_A(y(t))$.   

When $F(A_2)$ is a non-degenerate ellipse, the tangent lines to $z_0$ and $f_A(y)$ in $F(A_2)$ are the same as the tangent lines to those points in $F(A)$.  For all but one $t \in (0,1)$, this implies that the tangent lines to $z_0$ and $f_A(y)$ are not parallel.  From this we conclude that $x_0^*y \neq 0$, otherwise $x_0$ and $y$ would be distinct eigenvectors of the Hermitian matrix $\re (e^{-i\theta}A_2)$ for some rotation angle $\theta$, and their tangent lines would be parallel. Therefore the conditions of Lemma \ref{lem:2by2} apply, and show that $g$ is a selection of $f_A^{-1}$ on all of $F(A)$, except, perhaps, for the line segment from $z_0$ to the leftmost point in $F(A)$. However, the continuity of $g$ implies that $f_A(g(z)) = z$ must hold for all $z \in F(A)$.   
\end{proof}

When the boundary contains a corner point, it is easier to define a continuous selection. 
\begin{theorem} \label{thm:cornerpt}
Let $A \in M_n(\C)$ be a matrix such that $F(A)$ has no weak continuity failures, and there is a corner point $z_0 \in \partial F(A)$.  Then there is a continuous selection of $f_A^{-1}$ on $F(A)$.  
\end{theorem}
\begin{proof}
As a corner point, $z_0$ must be the image of a normal eigenvector $x_0$ under the action of $f_A$ \cite[Theorem 1]{donoghue1957}.  It is also the endpoint of two flat portions of the boundary.  Assume the other endpoints are $z_1$ and $z_{k-1}$ respectively.  Since there are no weak continuity failures, there is a continuous selection of the arc of the boundary from $z_1$ to $z_{k-1}$ opposite $z_0$.  Let $x(t)$ denote the vectors in this selection, with $t \in [1,k-1]$, so that $z(t) = f_A(x(t))$ parametrizes the arc of the boundary, and $z(1) = z_1$ and $z(k-1) = z_{k-1}$.  Since $x_0$ is a normal eigenvector, the compression of $A$ onto the subspace $\span \{x_0, x(t) \}$ must always be a normal matrix, with $x_0 \perp x(t)$ for all $t$.  Note that every $z \in F(A) \backslash \{z_0 \}$ can be written uniquely as 
$$
z = \lambda z_0 + (1-\lambda) z(t),
$$
where $\lambda \in [0,1)$ and $t \in [1,k-1]$ are continuous functions of $z$.  Let 
$$g(z) = \sqrt{\lambda} \, x_0 + \sqrt{1-\lambda} \, x(t).$$
By construction, $g$ is continuous on $F(A)\backslash \{z_0\}$, and if we define $g(z_0) = x_0$, then $g$ extends continuously to all of $F(A)$. Since $x_0$ is a normal eigenvector of $A$ and is orthogonal to $x(t)$, it follows that $f_A(g(z)) = z$ as desired, so $g$ is a continuous selection of $f_A^{-1}$ on $F(A)$.  
\end{proof}

Note that Theorem \ref{thm:cornerpt} covers the case when $A$ is any normal matrix.  

\section{Selections with Weak Continuity Failures}

\begin{theorem} \label{thm:weakcase}
Suppose that $A \in M_n(\C)$ and $f_A^{-1}$ is weakly continuous on $F(A)$ except at the points $w_1, \ldots, w_k \in \partial F(A)$. For any open set $G$ containing $\{w_1, \ldots, w_k\}$ there is a continuous selection of $f_A^{-1}$ on $F(A) \backslash G$.     
\end{theorem}
\begin{proof}
We will separate the proof into two cases. In the first case, suppose that $F(A)$ has no corner points. Rotate $F(A)$ so that there are no vertical flat portions and no exceptional point on the boundary has a vertical tangent.  Let $z_0$ denote the rightmost point in $F(A)$.  
As in the proof of Theorem \ref{thm:nocornermain}, we may construct a path $y:[0,1] \rightarrow \C S^n$ such that $f_A$ is a bijection from $y([0,1))$ to $\partial F(A)$, $y(0) \in f_A^{-1}(z_0)$, $y(1)$ is a scalar multiple of $y(0)$, and $y(0)^*y(t) \in \R$ for all $t \in (0,1)$. Unfortunately, $y(t)$ cannot be continuous at points corresponding to weak-continuity failures of $f_A^{-1}$, but we may construct $y(t)$ so that it is continuous everywhere else.  

Choose an $\epsilon$-neighborhood around each $w_j$ such that each neighborhood is contained in $G$ and each neighborhood only contains one exceptional point of $\partial F(A)$, namely the corresponding $w_j$.  In each neighborhood, chose $w_j^+$ and $w_j^- \in \partial F(A)$ on either side of $w_j$. There exist $t_j^+$ and $t_j^-$ such that $w_j^\pm = f_A(y(t_j^\pm))$.  We will replace $y(t)$ on the interval $[t_j^-,t_j^+]$ with an alternative path that is continuous in $\C S^n$.  

As in the proof of Theorem \ref{thm:nocornermain}, it will be convenient to let $x_0 = y(0)$. Let $u_j^\pm = y(t_j^\pm)$, and consider the 3-by-3 compression $A^{(j)}_3$ of $A$ corresponding to $\span \{x_0,u_j^+,u_j^- \}$.  By construction $F(A^{(j)}_3) \subset F(A)$. Since $A^{(j)}_3$ is only 3-by-3, the map $f_{A^{(j)}_3}^{-1}$ has no weak continuity failures on $F(A_3)$ \cite[Theorem 11]{CJKLS2}. For each $z \in F(A^{(j)}_3)$, the map $f_{A^{(j)}_3}^{-1}(z) \subset f_A^{-1}(z)$, since $f_{A^{(j)}_3}^{-1}$ takes values in $\span \{x_0, u_j^+,u_j^- \} \cap \C S^n$.  Thus, by finding a continuous selection of each $f_{A^{(j)}_3}^{-1}$ on the convex hull $\conv \{z_0, w_j^+, w_j^-\}$ that agrees with the construction of a continuous selection given in the proof of Theorem \ref{thm:nocornermain} on the line segments from $z_0$ to $w_j^+$ and $w_j^-$, we will find a continuous selection of $f_A^{-1}$ that applies to all of $F(A)$ except for the $\epsilon$-neighborhoods around the weak continuity failure points $w_j$.  In particular, the selection is continuous on $F(A)\backslash G$.

It is necessary to ensure that the map $f^{-1}_{A^{(j)}_3}$ is strongly continuous along the boundary of $F(A^{(j)}_3)$ from $w_j^-$ to $w_j^+$ opposite $x_0$.  By \cite[Theorem 11]{CJKLS2}, strong continuity holds at all points in $F(A^{(j)}_3)$, except possibly one exceptional point. This will only be a problem if that one point happens to be either $w_j^+$ or $w_j^-$, since in that case it may not be possible to find a continuous path $\gamma_j:[t_j^-,t_j^+] \rightarrow \C S^n$ such that $\gamma_j(t_j^\pm) = y(t_j^\pm)$, and such that $f_A(\gamma_j(t))$ parametrizes the arc of the boundary of $F(A^{(j)}_3)$ between $w_j^-$ and $w_j^+$.  

As shown in the proof of \cite[Theorem 11]{CJKLS2}, strong continuity fails for one point on the boundary of $F(A^{(j)}_3)$ if and only if $F(A^{(j)}_3)$ is a non-degenerate convex ellipse and $A^{(j)}_3$ has a normal eigenvalue on the boundary.  As mentioned previously, having a normal eigenvalue on the boundary of $F(A^{(j)}_3)$ would not prevent finding a continuous path $\gamma_j(t)$, unless the eigenvalue is either $w_j^\pm$. 

Suppose without loss of generality that this is the case, and that $w_j^+$ is a normal eigenvalue of $A^{(j)}_3$ located on the boundary of the elliptical range $F(A^{(j)}_3)$. In this case the 2-by-2 compression of $A$ onto $\span\{x_0, u_j^-\}$ is the same ellipse as $F(A^{(j)}_3)$. If the arc of the boundary of $F(A)$ from $w_j^+$ to $w_j$ coincides at infinitely many points with an elliptical arc of $\partial F(A^{(j)}_3)$, then the two critical curves are identical analytic curves. In that case, the arc of the boundary of $F(A)$ from $w_j^-$ to $w_j$ must be a different critical curve. By choosing a different $w_j^-$, we may ensure that the boundary of $F(A^{(j)}_3)$ does not coincide with the arc of the boundary of $F(A)$ from $w_j^+$ to $w_j$. Then by choosing a $w_j^+$ closer to $w_j$ if necessary, we may ensure that both $f_{A^{(j)}_3}^{-1}(w_j^\pm)$ are rank 1, and therefore strong continuity of $f_{A^{(j)}_3}^{-1}$ holds at both $w_j^\pm$. We may now construct a continuous selection of $f_{A^{(j)}_3}^{-1}$ on the arc of the boundary of $F(A^{(j)}_3)$ from $w_j^-$ to $w_j^+$ using \eqref{eq:p1} or \eqref{eq:p2} for round portions, and \eqref{eq:p3} for flat portions.  By change of variables, we may assume that the curve $\gamma_j$ we obtain has domain $[t_j^-,t_j^+]$ and by construction $\gamma_j(t_j^\pm) = u_j^\pm$.  

Repeating the argument from the proof of Theorem \ref{thm:nocornermain}, we use \eqref{eq:mainselect} with $y$ replaced by $\gamma_j$ to define the continuous selection of $f_{A^{(j)}_3}^{-1}$ and therefore $f_A^{-1}$ on $\conv \{z_0,w_j^-,w_j^+\}$. Note that we do not need to worry about $\gamma_j(t)$ having vertical tangent lines since the slopes of those tangent lines will be arbitrarily close to the slope of the tangent line to $w_j$ in $F(A)$.  Once we have a continous selection of $f_A^{-1}$ on each $\conv \{z_0,w_j^-,w_j^+\}$, the method of Theorem \ref{thm:nocornermain} extends those selections to a continuous selection of $f_A^{-1}$ on $F(A)\backslash G$.  

In the case where $F(A)$ has a corner point, the argument above can be simplified since we no longer need to worry if the points $w_j^\pm$ are strong continuity failure points of $f_{A^{(j)}_3}^{-1}$. We simply apply the technique of the proof of Theorem \ref{thm:cornerpt} directly to $F(A^{(j)}_3)$ to obtain a continuous selection of $f_{A^{(j)}_3}^{-1}$ on $\conv \{z_0, w_j^-, w_j^+\}$ which extends via the method of Theorem \ref{thm:cornerpt} to a continuous selection of $f_A^{-1}$ on all of $F(A)\backslash G$.  
\end{proof}

\begin{remark}
One might ask whether a continuous selection of $f_A^{-1}$ can be defined on $F(A)\backslash \{w_1, \ldots, w_k \}$, that is, everywhere except the points where weak continuity fails.  This is currently an open question.  

\end{remark}

\paragraph{\textbf{Acknowledgments.}}  We would like to thank Ilya Spitkovsky for his helpful advice and expertise. Thanks also to Charles Johnson for organizing the summer REU during which this project was started.

\bibliography{ContSelectRef}
\bibliographystyle{plain}

 \end{document}